\RequirePackage{fix-cm}
\RequirePackage{amsmath}
\documentclass{svjour3}                     
\smartqed  
\usepackage{graphicx}
\usepackage{float}
\usepackage{amsmath}
\usepackage{amsbsy}
\usepackage{amssymb}
\usepackage{graphicx}
\usepackage{amsfonts}
\usepackage{longtable}
\usepackage{color}
\usepackage{caption,setspace}
\begin{document}

\title{A  new nonmonotone adaptive trust region algorithm}


\author{Ahmad Kamandi\textsuperscript{a,*}, Keyvan Amini\textsuperscript{b} }


\institute{\textsuperscript{a} Department of Mathematics, University of Science and Technology of Mazandaran, P.O.Box 48518-78195, Behshahr, Iran. Email: ahmadkamandi@mazust.ac.ir.\\
\textsuperscript{b}Department of Mathematics,
	Faculty of Science, Razi University, P.O.Box 67141-15111
	Kermanshah, Iran. Email: kamini@razi.ac.ir}

\date{Received: date / Accepted: date}

\maketitle

\begin{abstract}

In this paper, we propose a new and efficient nonmonotone adaptive trust region algorithm to solve unconstrained optimization problems. This algorithm incorporates two novelties: it benefits from a radius dependent shrinkage parameter for adjusting the trust region radius that avoids undesirable directions and it exploits a new strategy to prevent sudden  increments of  objective function values in nonmonotone trust region techniques. Global convergence of this algorithm is investigated under  some mild conditions. Numerical experiments demonstrate the efficiency and robustness of the proposed algorithm in solving a collection of unconstrained optimization problems from the CUTEst package.

\end{abstract}

\section{Introduction}\label{sec1}
Consider the following unconstrained optimization problem:
\begin{equation}\label{e1} 
\begin{array}{ll}
\min          &~ f(x)\\
\mathrm{s.t.} &~ x\in\mathbb{R}^n,
\end{array}
\end{equation}
where $f:\mathbb{R}^n\rightarrow \mathbb{R}$ is a  differentiable function. We are interested in the case that the number of variables is large.
Despite  the fact that the well-known trust region method is  a well-documented  framework \cite{c1,n1}  in  numerical optimization for solving the Problem \eqref{e1}, its efficiency  needs to be improved. Since, itself or its variations are frequently required in tackling  emerged problems in extensive recent applications \cite{a3,c01,h1,s01}. 

In order to minimize $f(x)$, a trust region framework uses an approximation $x_k$ of a local minimizer to compute a trial step direction, $d_k$,  by solving the following subproblem:
\begin{equation}\label{e2} 
\begin{array}{ll}
\min          &~\displaystyle m_k(d)=f_k+g_k^T d+\frac{1}{2} d^TB_k d\\
\mathrm{s.t.} &~ \|d\|_2\leq \delta_k,
\end{array}
\end{equation}
where $f_k=f(x_k)$, $g_k=\nabla f(x_k)$, $\delta_k$ is a positive parameter that is called the trust region radius and $B_k$ is an approximation to the Hessian of the  objective function at $x_k$. In the rest of the paper $\|.\|$ denotes the Euclidean norm.

Finding a global minimizer of subproblem  \eqref{e2} is often too expensive such that, in practice, numerical methods are applied to find an approximation  \cite{g,m1,s3}.   Global  convergence of a classic trust region algorithm is proved provided that the  approximate solution $d_k$ of
subproblem \eqref{e2}  satisfies the following reduction estimation in the model function:
\begin{equation}\label{e3} 
 m_k(0)-m_k(d_k) \geq c \frac{1}{2} \|g_k\| \min \{ \delta_k , \frac{\|g_k\|}{\|B_k\|}\},
\end{equation}
with $c\in (0,1)$.

Given a fixed trial direction $d_k$, define the ratio $r_k$ as the following:
\begin{equation}\label{e4} r_k:=\frac{f_k-f(x_k+d_k)}{m_k(0)-m_k(d_k)}. \end{equation}

In  classical trust region methods, the $k$th iteration is called successful iteration if  $r_k>\mu$ for some $\mu\in(0,1)$. In this case, the trial point $x_k+d_k$ is accepted as a new approximation and the trust region radius is enlarged. Otherwise, the iteration $k$ is called an unsuccessful iteration; the trial point is rejected and the trust region is shrunk.
The efficiency of trust region methods strongly relies on the generated sequence of radii. A large radius possibly increases the cost of solving corresponding subproblem and a small radius increases the number of iterations. Hence, choosing an appropriate radius in each iteration is  challenging in trust region methods.  In an effort  to tackle this challenge, many authors have rigorously studied the adaptive trust region methods \cite{a1,e1,k1,s1,z1}. 

Zhang et al., in \cite{z1},   proposed the  following adaptive radius: 
\begin{equation}\label{e5} \delta_k=c^{p_k} \|g_k\|~\|\widehat{B}^{-1}_k\|, \end{equation}
where $c\in (0,1)$, $p_k$ is a nonnegative integer and
$\widehat{B}_k=B_k+E_k$ is a safely positive definite matrix based on a
modified Cholesky factorization from Schnabel and Eskow in \cite{s0}. Numerical results indicate that  embedding this adaptive radius in a pure trust region increases the efficiency. But, the formula \eqref{e5} needs to calculate the inverse matrix
$\widehat{B}^{-1}_k$ at each iteration such that it is not suitable  for large-scale problems.
 Shi and Guo, in \cite{s1}, proposed another adaptive radius  by
\begin{equation}\label{e6} \delta_k=-c^{p_k}\frac{g_k^Tq_k}{q_k^T\widetilde{B}_kq_k}\|q_k\|,\end{equation}
where $c\in (0,1)$, $p_k$ is a nonnegative integer and $q_k$ is a
vector satisfying
\begin{equation}\label{e7} -\frac{g_k^Tq_k}{\|g_k\|.\|q_k\|} \geq \tau,\end{equation}
with $\tau \in (0,1]$. Moreover, $\widetilde{B}_k$ is generated by the
procedure: $\widetilde{B}_k = B_k+iI$, where $i$ is the smallest
nonnegative integer such that $q_k^T\widetilde{B}_kq_k >0$.
  It is simple to see that the radius \eqref{e6}, for $p_k=0$, estimates the norm of exact minimizer of the quadratic model  $f_k+g_k^T d+\frac{1}{2} d^T\widetilde{B}_k d$ along the direction $q_k$. 
  
   Motivated by this adaptive radius, Kamandi et al.  proposed an efficient adaptive trust region method in which the radius at each iteration is determined by using the information gathered from the previous step \cite{k1}. Let $d_{k-1}$ be the solution of the subproblem in the previous step, for parameters $\tau \in (0,1)$ and $\gamma>1$ define: 
\begin{equation}\label{e8}
q_k:=\left\{
\begin{array}{ll}
-g_k &~~ \textrm{if}~ k=0 ~\textrm{or}~ 
\displaystyle \frac{-(g_k^T d_{k-1})}{\|g_k\| \|d_{k-1}\|}\leq \tau ,\\\\
d_{k-1} &~~ o.w,
\end{array}  
\right.
\end{equation}
and 
\begin{equation}\label{e9}
s_k:=\left\{
\begin{array}{ll}
\displaystyle -\frac{g_k^Tq_k}{q_k^TB_kq_k}~\|q_k\| & ~~ \mathrm{if}~ k=0,\\
\displaystyle \max \left\{-\frac{g_k^Tq_k}{q_k^TB_kq_k}~\|q_k\|,~ \gamma \delta_{k-1} \right\}& ~~ \mathrm{if}~ k \geq 1,
\end{array}  
\right.
\end{equation}
Then,  the proposed algorithm in \cite{k1} for solving \eqref{e1} is as follows: 

\rule{10cm}{.2pt}

\textbf{(IATR) Improved adaptive trust-region algorithm }

\rule{10cm}{.2pt}

\textbf{Input:} $x_0 \in \mathbb{R}^n$, a positive definite matrix  $B_0\in   \mathbb{R}^{n\times n}$,$\overline{\delta}>0$,

\hspace{10mm} $c,\mu\in (0,1)$, $\tau \in (0,1)$,   $\gamma\geq1$ and $\epsilon>0$.
  
\textbf{Begin}

\hspace{4mm} $k \leftarrow 0$, compute $f_0$ and $g_0$.

\hspace{4mm} \textbf{While} ($\|g_k\| > \epsilon$)

\hspace{10mm} Compute $q_k$ by (\ref{e8})  and $s_k$ by (\ref{e9}),

\hspace{10mm} Set $\delta_{k_0}=\min\{s_k,\overline{\delta}\}$,

\hspace{10mm} $r_k \leftarrow 0 $, $p\leftarrow 0 $,

\hspace{10mm} \textbf{While} ($r_k<\mu$)

\hspace{16mm} $\delta_{k_p} \leftarrow c^p \delta_{k_0}$,

\hspace{16mm} Compute $d_{k_p}$ by solving (\ref{e2}) with radius $\delta_{k_p} $ ;~ compute $r_k$ by (\ref{e4}),

\hspace{16mm} $p\leftarrow p+1$.

\hspace{10mm} \textbf{End While}

\hspace{10mm} $x_{k+1} \leftarrow x_k+d_{k_p}$,

\hspace{10mm} Update $B_k$ by a quasi-Newton  formula,

\hspace{10mm} $k \leftarrow k+1$.

\hspace{4mm} \textbf{End While}

\textbf{End}

\rule{10cm}{.2pt}

Despite it enjoys many advantages \cite{k1}, this algorithm has several disadvantages. First, setting a fixed value for the shrinkage parameter '$c$', in the inner  loop of the IATR algorithm, is not an intelligent choice. In order to see this, suppose  that the step direction $d_{k_0}$,  the solution of the subproblem \eqref{e2} with the radius $\delta_{k_0} $, is rejected by the ratio test.  In this case, the algorithm shrinks  the radius $\delta_{k_0}$ by the factor $c\in(0,1)$. Hence, we have the new following subproblem:
 
  \begin{equation}\label{e010}
 \begin{array}{ll}
\min~&m_k(d)=f_k+g_k^Td+\frac{1}{2}d^TB_kd \\
&\|d\|\leq c\delta_{k_0}.
\end{array}
\end{equation}
Since $c\delta_{k_0} <\delta_{k_0}$, it is clear that the feasible region of subproblem \eqref{e010} is a subset of the feasible region of subproblem \eqref{e2}. So, in case that $\|d_{k_0}\| \leq c\delta_{k_0}$, $d_{k_0}$ is also a solution of \eqref{e010}, although we know that it is rejected by the  ratio test. This means that the new step direction is rejected by the ratio test again without any improvement; solving the new subproblem has redundant computational costs though.

  Another drawback of a constant shrinkage parameter occurs   when  the trust region radius is too large and the shrinkage parameter is close to one: the algorithm is forced to solve the trust region subproblem several times until it finds a successful step. So, using a shrinkage parameter close to one may increase the number of function evaluations. On the other hand, using a small shrinkage parameter may cause to shrink the trust radius too fast; in this case, the number of iterations  increases. 

Furthermore,  the sequence of function evaluations generated by this algorithm is decreasing and numerical results show that imposing monotonicity to trust  region algorithms may reduce the speed of convergence for some problems, specially in the presence of narrow valley. In order to overcome similar drawbacks, Grippo et al. proposed a nonmonotone line search technique for Newton's method \cite{g1}. By generalizing the technique to the trust region methods, nonmonotone version of these methods appeared in the literature \cite{a1,a2,d0,p1,s2,z0}. 

 The basic difference between monotone and nonmonotone trust region approaches is due to the definition of the ratio $r_k$. In a nonmonotone trust region, the ratio is defined by
\begin{equation}\label{e10} \hat{r}_k=\frac{C_k-f(x_k+d_k)}{m_k(0)-m_k(d_k)}, \end{equation}
where $C_k$ is  a parameter greater than or equal to $f_k$. In this paper, we call $C_k$, the nonmonotone parameter. In different versions of nonmonotone algorithms,  the nonmonotone parameter is computed based on different methodologies. A common parameter for nonmonotone trust region methods is
\begin{equation}\label{e11}
 f_{l_k}:=\max _{0\leq j \leq  N_k}  \lbrace f_{k-j} \rbrace ,
 \end{equation} 
 where $N_0=0$ and $N_k=\min \{k ,N\} $ for a fixed integer number  $N\geqslant 0$. 
 
 Remark that by taking maximum in the parameter \eqref{e11},  a potentially very good function value can be excluded. Trying to tackle this drawback, 
Ahookhosh et al. in \cite{a2} proposed  the following nonmonotone parameter
\begin{equation}\label{e12}
 R_{k}=\eta_{k} f_{l_k}+(1-\eta_{k})f_k,
\end{equation}
where $ \eta_{k} \in  [\eta_{min},\eta_{max}],\eta_{min} \in [0,1) $ and $ \eta_{max} \in [\eta_{min},1] $. When  $\eta_k$ is close to one the effect of nonmonotonicity is amplified. On the other hand, when $\eta_k$ is close to zero the algorithm ignores the effect of the term $f_{l_k}$ and behaves monotonically. 

 In 2019,  Xue et. al. proposed a nonmonotone version of the IATR algorithm based on the nonmonotone parameter \eqref{e12} \cite{x1}. They also used a scaled memoryless BFGS  formula to update the approximation of the Hessian matrix.   By analyzing the numerical behavior of  nonmonotone versions of IATR using the aforementioned nonmonotone parameters, we find out that  in some problems, for example OSCIGRAD, the difference between  the current objective value $f_k$ and the nonmonotone parameter becomes too large and in this case a large increase is allowed to happen in the next iteration. Another drawback of above nonmonotone parameters is that they strongly depend on the choice of the memory parameter $N_k$ and the parameter $\eta_k$ and there is no specific rule to adjust them.

In this paper, by combining the idea of adaptive trust region and nonmonotone techniques, we propose a new efficient nonmonotone trust region algorithm for solving unconstrained optimization problems. In the new algorithm, a radius dependent shrinkage parameter is used to adjust the trust region radius in rejected steps that addresses the first disadvantage of IATR.  For resolving the second disadvantage, a novel strategy is used to compute the nonmonotone parameter in this algorithm which prevents a sudden increment in the objective values.

 The paper is organized as follows:  the new algorithm is proposed in the next section. Section 3 is devoted to its convergence properties. The numerical results of testing the new algorithm to solve a collection of the CUTEst test problems are reported in section 4. The last section includes the conclusion.

\section{The new algorithm}\label{sec2}

In this section, we propose our algorithm for solving unconstrained optimization problems.  

As mentioned in the previous section, setting a fixed value for the shrinkage parameter $c$ in the inner loop of the IATR algorithm may impose some useless computational costs to this algorithm. Therefore, for resolving this issue, we propose the following radius dependent shrinkage parameter
 \begin{equation}\label{e13}
  c_{k_p}:=c(\delta_{k_p} ),
 \end{equation}
 where $c(\delta):(0,\bar{\delta}]\rightarrow [\alpha_0,\alpha_1]$ is a decreasing function where $0<\alpha_0<\alpha_1<1$ and $\bar{\delta}$ is the maximum possible radius.  Also, in order to exclude the rejected trial step $d_{k_p}$, in the new algorithm we define the new radius as 
 $\delta_{k_{p+1}}=c_{k_p}\|d_{k_p}\|$.
 
Note that  the radius dependent parameter \eqref{e13} is close to $\alpha_0$ for a large trust region radius and is close to $\alpha_1$ for a small one. Hence, this parameter shrinks the trust region harshly  for large trust region radii and helps the new algorithm to find a successful step direction fast enough. Further, it shrinks the trust region mildly for the case that the trust region radius is small. 
 
Also numerical tests persuaded us to consider a radius dependent parameter $\gamma_k=\gamma(\delta_{k-1})$ based on the previous trust region radius and use it instead of constant parameter $\gamma$ in \eqref{e9}.  Similar to \eqref{e13} $\gamma(\delta_{k-1})$ is a decreasing function bounded from below by 1.
 
With the goal of overcoming the second disadvantage of the IATR algorithm and  building an efficient nonmonotone version of it, we propose a new nonmonotone parameter $C_k$. This new parameter benefits from  nonmonotonicity in an adaptive way  compared to the mentioned parameters.  When a very good function value is found at iteration $k$, it is better to save that by forcing  the algorithm to behave monotonically for the next iteration.  To this aim, we define the new nonmonotone parameter  using not only the simple parameter $f_{l_k}$  defined by \eqref{e11} but also considering its relative difference from the current function value. 

For a positive parameter $\nu$, define sequences  $\{M_k\}$  and $\{I_k\}$ as follows:
\begin{equation*}
M_k:=\left\{
\begin{array}{ll}
0 &~~ \textrm{if}~ k=0~ or~~f_{l_k}-f_k>\nu |f_k| \\
M_{k-1}+1 &~~ o.w,
\end{array}  
\right.
\end{equation*}
and
\begin{equation*}
I_k:=\left\{
\begin{array}{ll}
0 &~~ \textrm{if}~ k=0~ or~~f_k< f_{k-1} \\
I_{k-1}+1 &~~o.w.
\end{array}  
\right.
\end{equation*}
Having above sequences, for  fixed natural numbers  $\bar{N}$ and $\bar{I} $, we define the new nonmonotone parameter $C_k$ as
 \begin{equation}\label{e14}
C_k:=\left\{
\begin{array}{ll}
\max _{0\leq j \leq  n_k}  \lbrace f_{k-j} \rbrace  & \textrm{if}~ I_k\leq \bar{I} \\
f_k& o.w,
\end{array}  
\right.
 \end{equation} 
 where  $n_k=\min \{M_k ,\bar{N}\} $. Note that the sequence $\{I_k\}$ counts the number of consecutive increments in the objective function values. So, the nonmonotone parameter $C_k$ defined by \eqref{e14} prevents large increments in the objective function values and guarantees  at least one decrease for each $\bar{I}$th iteration.
Also, the definition of the sequence $\{M_k\}$ makes the new algorithm monotone when the relative difference between $f_{l_k}$ and the current function value is large and  prevents a sudden increment in  the objective function values for the next iteration.

Now, we are ready to propose the new adaptive nonmonotone trust region algorithm.

\rule{10cm}{.2pt}
 
\textbf{(NATR) Nonmonotone adaptive trust-region algorithm }

\rule{10cm}{.2pt}
\begin{itemize}
\item[] {\textbf{Input:} {$x_0 \in \mathbb{R}^n$, a positive definite matrix 
 $B_0\in   \mathbb{R}^{n\times n}$, $\overline{\delta}>0$, 
 
 a decreasing function $c(\delta)$,
  $\mu\in (0,1)$, $\tau \in (0,1)$, 
  $\gamma\geq1$ and $\epsilon>0$}.}
  
\item[] {\textbf{Begin}}

\item[] {\hspace{4mm} $k \leftarrow 0$; compute $f_0$ and $g_0$.}

\item[] {\hspace{4mm}\textbf{While} ($\|g_k\| > \epsilon$)}

\item[] { \hspace{10mm} Compute $q_k$ by \eqref{e8} and $s_k$ by \eqref{e9}, }

\item[] {\hspace{10mm} Set $\delta_{k_0}=\min\{s_k,\overline{\delta}\}$,}

\item[] {\hspace{10mm} Compute $C_k$ by \eqref{e14},}

\item[] {\hspace{10mm}   Compute  $d_{k_0}$ by solving  \eqref{e2} with radius $\delta_{k_0}$   and $\hat{r}_k$ by \eqref{e10} and set

\hspace{10mm} $p=0 $.}

\item[] {\hspace{10mm}\textbf{While} ($\hat{r}_k<\mu$)}

\item[] {\hspace{16mm}   Compute  $c_{k_p}$  by \eqref{e13},}

\item[] {\hspace{16mm} $\delta_{k_{p+1}}=c_{k_p}\|d_{k_p}\|$,}

\item[] {\hspace{16mm} Compute $d_{k_{p+1}}$ by solving \eqref{e2}  with radius $\delta_{k_{p+1}}$ and $\hat{r}_k$ by \eqref{e10},}

\item[] {\hspace{16mm} $p\leftarrow p+1$.}

\item[] {\hspace{10mm}\textbf{End While}}

\item[] {\hspace{10mm} $x_{k+1} \leftarrow x_k+d_{k_p}$,}

\item[] {\hspace{10mm} Update $B_k$ by a quasi-Newton formula,}

\item[] {\hspace{10mm} $k \leftarrow k+1$.}

\item[] {\hspace{4mm} \textbf{End While} }

\item[] {\textbf{End}}
\end{itemize}
\rule{10cm}{.2pt}

In the next section, we propose the convergence properties of the new algorithm.

\section{Convergence properties}\label{sec3}
In this section, we analyze the global convergence of the new algorithm. To
this end, we need the following assumptions:
\begin{description}
\item[\textbf{(H1)}] The objective function $f(x)$ is continuously
differentiable and has a lower bound on the level set 
\[
\mathcal{L}(x_0)=\{x\in \mathbb{R}^n ~|~ f(x)\leq f(x_0),\ x_0\in \mathbb{R}^n\}.
\]
\item[\textbf{(H2)}] The approximation matrix $B_k$ is
uniformly bounded, i.e., there exists a constant $M>0$ such that
\[
\|B_k\|\leq M,~~~ \mathrm{for~ all}~ k \in \mathbb{N}.
\]
\end{description}

The following lemma is  similar for both the IATR  and the NATR algorithms, so its proof is omitted.
\begin{lemma}\label{lem1}
Suppose that the sequence $\{x_k\}$ is generated by the NATR algorithm, then
\begin{equation*}
|f(x_k+d_{k_p})-m_k(d_{k_p}))|= o(\|d_{k_p}\|).
\end{equation*}
\end{lemma}
\begin{proof}
See \cite{n1}. $\square$
\end{proof} 
 The next two lemmas guarantee the existence of  some lower bounds for the  trust region radius $\delta_{k_0}$ and the norm of the trial step $d_{k_0}$   at iteration $k$ generated by the NATR algorithm. 
  \begin{lemma}\label{lem2} Suppose that  $\delta_{k_0}=\min\{s_k, \bar{\delta} \}$ is  the trust region radius at iteration $k$ of the NATR algorithm such that $s_k$ is defined by  \eqref{e9}. Then,

\begin{equation}\label{e15}
\delta_{k_0}\geq \min\{\tau \frac{\|g_k\|}{\|B_k\|} , \bar{\delta}\}.
\end{equation}
\end{lemma}

\begin{proof} 
In case $\bar{\delta} \leq s_k $, we have $\delta_{k_0}=\bar{\delta}$ and the inequality \eqref{e15} is valid. Thus, consider the case that $ s_k< \bar{\delta}$  and $\delta_{k_0}= s_k$.
The definition of $s_k$ in \eqref{e9}, and the Cauchy-Schwarz inequality yield that
\begin{equation}\label{e16}
\delta_{k_0} \geq \frac{-g_k^Tq_k}{\|B_k\|\|q_k\|}.
\end{equation}
By the definition of $q_k$ in \eqref{e8} if $q_k=-g_k$,  the inequality \eqref{e16} results in
\begin{equation*}
\delta_{k_0} \geq \frac{\|g_k\|}{\|B_k\|}.
\end{equation*}
When $q_k=d_{k-1}$,  we  have $-g_k^Tq_k \geq \tau \|g_k\| \|q_k\| $ such that the inequality \eqref{e16} implies that
 \begin{equation*}
\delta_{k_0} \geq \tau \frac{\|g_k\|}{\|B_k\|}.
\end{equation*}
By the above explanation  along with the fact that $\tau \in (0,1)$ we can conclude that \eqref{e15} is valid. $\square$
 \end{proof} 
\begin{lemma}\label{lem3} Suppose that  $d_{k_0}$ is  the solution of subproblem \eqref{e2} with radius $\delta_{k_0}$. Then,
\begin{equation}\label{e17}
\|d_{k_0}\| \geq  \min \{ \frac{\|g_k\|}{\|B_k\|} , \delta_{k_0}\}.
\end{equation}
\end{lemma}
\begin{proof}
 By Theorem 4.1 of \cite{n1}, when $d_{k_0}$ lies strictly inside the feasible region of subproblem \eqref{e2}, we must have $B_kd_{k_0}=-g_k$ such that the Cauchy-Schwarz inequality yields that
\begin{equation*}
\|d_{k_0}\| \geq \frac{\|g_k\|}{ \|B_k\|}.
\end{equation*}
In the other case $d_{k_0}$ lies on the boundary of the feasible region of subproblem \eqref{e2} which implies $\|d_{k_0}\|=\delta_{k_0}$.  So, from the above discussion we can conclude that \eqref{e17} is valid.$\square$

\end{proof} 
 By  \eqref{e15} and \eqref{e17}, we can also obtain a lower bound for $\delta_{k_p}$. Note that, at iteration $k$ of the NATR algorithm, for any $p \geq 1$, the solution $d_{k_p}$ lies on the boundary of the region defined by $\delta_{k_p}$. Since the objective is fixed for each iteration, when the trial step $d_{k_p}$ is rejected by the ratio test, the new radius $\delta_{k_{p+1}}$ is set to exclude  $d_{k_p}$ from the new region. Thus, by the contraction of the inner loop of the NATR algorithm, we have 
\begin{equation*}
\begin{split}
 \delta_{k_p}&= c_{k_{p-1}} \| d_{k_{p-1}}\| =c_{k_{p-1}}\delta_{k_{p-1}} \\
&=c_{k_{p-1}}c_{k_{p-2}}\| d_{k_{p-2}}\|=c_{k_{p-1}}c_{k_{p-2}} \delta_{k_{p-2}} \\
&\ \vdots \\
&=\prod_{i=0}^{p-1} c_{k_i}\|d_{k_0}\|.
\end{split}
\end{equation*}

This equation along with  \eqref{e15}, \eqref{e17} and the fact that $\alpha_0$  is a lower bound and $\alpha_1$ is an upper bound for $c_{k_i}$, for any $i \geq 0$, yield that 
\begin{equation}\label{e18}
 \alpha_0^p  \min \{ \tau \frac{\|g_k\|}{\|B_k\|} , \bar{\delta} \} \leq \delta_{k_p} \leq \alpha_1^p \bar{\delta}.
\end{equation}

In Lemma \ref{lem4}, we propose a lower bound for the denominator of the ratio  $\hat{r}_k$ defined by \eqref{e10} which is used  in Lemma \ref{lem5} to prove that  the inner loop of the NATR algorithm terminates in a finite number of inner iterations.
\begin{lemma}\label{lem4}
Suppose that (H2) holds, the sequence $\{x_k\}$ is
generated by the NATR algorithm, and $d_{k_p}$ is an approximate solution of the subproblem 
\eqref{e2} with radius $\delta_{k_p}$, that satisfies \eqref{e3}. Then, 
\begin{equation}\label{e19}
 m_k(0)-m_k(d_{k_p}) \geq \frac{1}{2} c  \alpha_{0}^p \|g_k\| \min \{ \tau \frac{\|g_k\|}{M} , \bar{\delta}\} ,~~~\forall k \in
\mathbb{N}.
\end{equation}
for all $k \in \mathbb{N}$.
\end{lemma}
\begin{proof}
By \eqref{e3}, for $d_{k_p}$, we have
\begin{equation*}
 m_k(0)-m_k(d_{k_p}) \geq c \frac{1}{2} \|g_k\|   \min \{  \delta_{k_p} , \frac{\|g_k\|}{\|B_k\|}\}.
\end{equation*}
This inequality along with  assumption (H2) and the inequality \eqref{e18} result in 
\begin{equation*}
 m_k(0)-m_k(d_{k_p}) \geq \frac{1}{2} c  \alpha_{0}^p \|g_k\| \min \{ \tau \frac{\|g_k\|}{M} , \bar{\delta}\}.
\end{equation*}
So, the proof is completed. $\square$
\end{proof} 
\begin{lemma}\label{lem5}
Suppose that assumption (H2) holds, then the inner loop of the NATR algorithm is
well-defined.
\end{lemma}
\begin{proof}
By contradiction, assume that the inner loop of the NATR algorithm at iteration $k$ is not well-defined. Since $x_k$ is not the optimum, $\|g_k\|
\geq \epsilon$. 

Now, let $d_{k_p}$ be the solution of subproblem (\ref{e2})
corresponding to $p\in \mathbb{N}\cup\{0\}$ at $x_k$. It follows from Lemma \ref{lem1} and (\ref{e19}) that
\begin{equation*}\begin{split}
\left|\frac{f(x_k)-f(x_k+d_{k_p})}{m_k(0)-m_k(d_{k_p})}-1\right|
&=\left|\frac{f(x_k)-f(x_k+d_{k_p})-(m_k(0)-m_k(d_{k_p}))}{m_k(0)-m_k(d_{k_p})}\right|\\
& \leq\frac{o(\|d_{k_p}\|)}{\frac{1}{2} c  \alpha_{0}^p \|g_k\| \min \{ \tau \frac{\|g_k\|}{M} , \bar{\delta}\}  }
\\
& \leq \frac{o(\|d_{k_p}\|)}{\frac{1}{2} c  \alpha_{0}^p\epsilon  \min \{ \tau \frac{\epsilon}{M} , \bar{\delta}\}}.
\end{split}\end{equation*}
By \eqref{e18}, we have $\delta_{k_p} \leq \alpha_1^p \bar{\delta}$. So, if  the inner loop of the NATR algorithm cycles infinitely many times (or $p \rightarrow \infty$), then $\delta_{k_p}$ tends to zero. Thus, the feasibility of $d_{k_p}$, $\|d_{k_p}\|\leq\delta_{k_p}$,  implies that the right hand side of the above equation tends to zero. This means that for sufficiently large $p$, we get
\begin{equation*}
\frac{f(x_k)-f(x_k+d_{k_p})}{m_k(0)-m_k(d_{k_p})} \geq \mu,
\end{equation*}
This inequality along with the fact that $C_k \geq f_k$ yield that
\begin{equation*} \hat{r}_k=\frac{C_k-f(x_k+d_{k_p})}{m_k(0)-m_k(d_{k_p})} \geq \mu,
\end{equation*}
which means that the inner cycle of the NATR algorithm is terminated in 
the finite number of internal iterations. $\square$
\end{proof} 
Two following lemma illustrates some properties of the sequences $\{x_k\}$ and $\{C_k\}$, generated by the  NATR algorithm. The result of this lemma is used to prove the global convergence of the NATR algorithm.
\begin{lemma}\label{lem6}
Suppose that  assumption (H1) holds and the sequence $\{x_k\}$ is
generated by NATR algorithm. Then, $f_{k+1}\leq C_{k+1} \leq  C_k$ . Therefore,  the sequence  $\{x_k\}$ is contained in the level set $\mathcal{L}(x_0)$ and the sequence $\{C_k\}$ is convergent.
\end{lemma}
\begin{proof}
By  the NATR algorithm, at successful iteration $k$, we have
\begin{equation*}
 C_k-f_{k+1} \geq \mu (m_k(0)-m_k(d_{k_p})) \geq 0. 
\end{equation*}
This inequality  along with \eqref{e14} and the definition of the sequence $\{M_k\}$ imply that
\begin{equation*} 
C_{k+1}= \max _{0\leq j \leq  n_{k+1}}  \lbrace f_{k+1-j} \rbrace \leq \max \{f_{k+1} , C_k \} = C_k.
\end{equation*}
Thus,
 \begin{equation}\label{e20} 
f_{k+1} \leq  C_{k+1}  \leq C_k \leq C_0=f_0.
\end{equation}
The last equation means that $\{x_k\}$ is contained in the level set $\mathcal{L}(x_0)$. Accordingly, assumption (H1)  and \eqref{e20} yield that $\{C_k\}$ is  decreasing and bounded from below.  Therefore,  the sequence $\{C_k\}$ is convergent. $\square$
\end{proof} 

Now, we are ready to present the global convergence theorem.
\begin{theorem}
Suppose that assumptions (H1) and (H2) hold. Then, the NATR algorithm 
 either terminates in a finite number of steps, or generates
an infinite sequence $\{x_k\}$ such that
\begin{equation}\label{e21}\liminf_{k\rightarrow\infty}\|g_k\|=0.\end{equation}
\end{theorem}
%
\begin{proof}
If the NATR algorithm  terminates in a finite  number of steps, then the proof is trivial.
Hence,  assume that the sequence $\{x_k\}$  generated by this algorithm is  infinite, we show that (\ref{e21}) holds. To this end, suppose that 
there exists a constant $\epsilon_0>0$  such that
\begin{equation}\label{e22}
\|g_k\| \geq \epsilon_0
\end{equation}
for all $k$. 
Let $C_k=f_{i_k}$,  where 
$f_{i_k}=\arg max \{\max _{0\leq j \leq  n_{k}}  \lbrace f_{k-j} \}\}$ . Then, by lemma \ref{lem6}, the sequence $\{f_{i_k}\}$ is a convergent subsequence of $\{f_k\}$. 
By the fact that $\hat{r}_k\geqslant\mu$,
we have 
\begin{equation*}
f_{i_k}-f_{k+1} \geq  \mu (m_k(0)-m_k(d_{k_p})).
\end{equation*}
Next, by replacing $k$ with $i_k-1$, we conclude that
\begin{equation*}
f_{i_{i_k-1}}-f_{i_k} \geq  \mu (m_{i_k-1}(0)-m_{i_k-1}(d_{(i_k-1)_p})).
\end{equation*}
This inequality  along with lemma \ref{lem4} yield that
\begin{equation*}
\begin{array}{ll}
f_{i_{i_k-1}}-f_{i_k} & \geq  \mu [\frac{1}{2} c  \alpha_{0}^p \| g_{i_k-1}\| \min \{ \tau \frac{\| g_{i_k-1}\|}{M} , \bar{\delta}\} ] \\
& \geq  \mu [ \frac{1}{2} c  \alpha_{0}^p \epsilon_0 \min \{ \tau \frac{\epsilon_0}{M} , \bar{\delta}\}].
\end{array}
\end{equation*}
 Taking limit from this inequality  when $k\rightarrow \infty$
implies that
\begin{equation*}
0\geq   \mu [\frac{1}{2} c  \alpha_{0}^p \epsilon_0 \min \{ \tau \frac{\epsilon_0}{M} , \bar{\delta}\}],
\end{equation*}
which is a contradiction. Thus, the equation \eqref{e21} is valid. $\square$
\end{proof}

\section{Numerical results}\label{sec4}
In this part of the paper, we report some numerical experiments that indicate the efficiency of the proposed algorithm. The results were obtained from  implementing  two versions of the NATR algorithm and  the adaptive nonmonotone algorithm proposed by Xue et. al. \cite{x1}  in MATLAB environment on a laptop (CPU Corei7-2.5 GHz, RAM 12 GB) and comparing their results of solving a collection of 228 unconstrained optimization test problems from the CUTEst collection \cite{g0}. The test problems and their dimensions are listed in table 1.

 In this  section, we use the following notations:
\begin{itemize}
  \item \textbf{AINTR}:  The adaptive nonmonotone algorithm proposed by Xue et. al. \cite{x1}.
   \item \textbf{NATR1}:  Nonmnotone adaptive trust region method (the NATR Algorithm) based on  the modified BFGS update formula used in \cite{k1}.
  \item \textbf{NATR2}:  Nonmnotone adaptive trust region method (the NATR Algorithm) based on the scaled memoryless BFGS update formula used in \cite{x1}.
 \end{itemize}

For the NATR Algorithm we used the following radius dependent parameters:
\begin{equation*}
\gamma(\delta)= \left\{
\begin{array}{ll}
1.5 &~~ \textrm{if}~\frac{\bar{\delta}}{2} <\delta \leq \bar{\delta} \\
1.9 &~~ \textrm{if}~\frac{\bar{\delta}}{5} <\delta \leq \frac{\bar{\delta}}{2}\\
2   &~~ \textrm{if}~\frac{\bar{\delta}}{10} <\delta \leq \frac{\bar{\delta}}{5}\\
3   &~~ \textrm{if}~ 10^{-6} <\delta \leq \frac{\bar{\delta}}{10}\\
3.5 &~~ \textrm{if}~ o.w\\
\end{array}  
\right.
\end{equation*}
and
\begin{equation*}
c(\delta)= \left\{
\begin{array}{ll}
0.3 &~~ \textrm{if}~\frac{\bar{\delta}}{10} <\delta \leq \bar{\delta} \\
0.45  &~~ \textrm{if}~ 10^{-6} <\delta \leq \frac{\bar{\delta}}{10}\\
0.6 &~~ \textrm{if}~ o.w\\
\end{array}  
\right.
\end{equation*}
similar to \cite{x1}, the other parameters are chosen as $\tau=0.01$, $N=15$, $\mu=0.07$, $\bar{\delta}=100$,  $\epsilon=10^{-6}\|g_0\|$ and for the NATR algorithm the remaining parameters are selected as $\bar{N}=10$, $\bar{I}=6$ and $\nu=10$. The trust region subproblems are solved by the Steihaug-Toint scheme \cite{s3}.

To visualize the whole behaviour of the algorithms, we use the performance profiles proposed by
Dolan and More \cite{d1}.The results of 14 test problems (the red ones in the table) are excluded from comparison because all the tested algorithms failed to solve them. So, the comparison of the algorithms is based on the remaining 214 test problems. Among these 214 test problems NATR1, NATR2 and AINTR faced with 9, 42, 49 failure(s) respectively. 

The total number of function evaluations, the total number of iterations  and the running time of each algorithm are considered  as performance indexes. Note that, at each iteration of the considered algorithms, the gradient of objective function is computed just one time, so the total number of iterations and the total number of the gradient evaluations are the same. 
 Figure 1 illustrates the performance profile of these algorithms, where the performance index is the total number of function evaluations. It can be seen that the NATR1 is the best solver with probability around 55\%, while the probability of solving a problem as the best solver is around 42\% and 30\% for NATR2 and AINTR respectively.

 The performance index in Figure 2 is the total number of iterations. From this figure, we observe that NATR1 obtains the most wins on approximately 58\% of all test problems and the probability of being the best solver is 41\% and 29\%  for NATR2 and AINTR respectively.  
 
  The performance profiles for the running times are illustrated in Figure 3. From this figure, it can be observed that NATR1 is the best algorithm. Another important factor of these three figures is  that the graph of NATR1 algorithm grows up faster than the others. 
 
 From the presented results, we can conclude that the radius dependent shrinkage parameter and the new nonmonotone procedure are effective to improve the efficiency of the IATR algorithm \cite{k1} compared with the nonmonotone algorithm proposed by  Xue \cite{x1}.

\renewcommand{\arraystretch}{1.1}
{\small
\begin{longtable}{llll}
\multicolumn{2}{l}%
{Table 1. List of test problems}\\[5pt]
\hline 
\multicolumn{1}{l}{Problem name} &\multicolumn{1}{c}{Dim}&
\multicolumn{1}{l}{Problem name} &\multicolumn{1}{c}{Dim}\\
\hline
\endfirsthead
\multicolumn{2}{l}%
{Table 1. (\textit{continued})}\\[5pt]
\hline
\endhead
\hline
\endfoot
\endlastfoot
 ARGLINA & 100, 200 & ARGLINB & 100, 200\\
ARGLINC & 100, 200  & BDQRTIC & 100, 500, 1000, 5000\\
BROWNAL & 100, 200, 1000 & BRYBND  & 100, 500\\
CHAINWOO & 100& CURLY10 & 100\\
CURLY20& 100& CURLY30& 100\\
EIGENALS & 110,\textcolor{red}{ 2550} & EIGENBLS &110, \textcolor{red}{2550}\\
EIGENCLS & 462, \textcolor{red}{2652} & EXTROSNB & 100,1000\\
FREUROTH & 100, 500, 1000, 5000&GENROSE & 100, 500\\
LIARWHD & 100, 500, 1000, 5000 & MANCINO &100\\
MODBEALE &200, 2000& MSQRTALS &100, 529 \\
MSQRTBLS  &100, 529 & NONDIA  & 100, 500, 1000, 5000\\
NONSCOMP & 100, 500, 1000, 5000& OSCIGRAD & 100, 1000\\
OSCIPATH &100, 500& PENALTY1&100, 500, 1000\\
PENALTY2&100, 200& SENSORS &100, 1000\\
SPMSRTLS & 100, 499, 1000, 4999& SROSENBR & 100, 500, 1000, 5000\\
SSBRYBND & 100 & TQUARTIC& 100, 500, 1000, 5000\\
VAREIGVL & 100, 500, 1000, 5000 &WOODS  &100, 1000, 4000\\
ARWHEAD  & 100, 500, 1000, 5000 &BOX    &100, 1000\\
BOXPOWER  &100, 1000& COSINE  &100, 1000\\
CRAGGLVY &100, 500, 1000, 5000&TESTQUAD &\textcolor{red}{1000, 5000}\\
DIXMAANA &300, 1500, 3000 &DIXMAANC &300, 1500, 3000\\
DIXMAAND &300, 1500, 3000 &DIXMAANE &300, 1500, 3000\\
DIXMAANF &300, 1500, 3000 &DIXMAANG &300, 1500, 3000\\
DIXMAANH &300, 1500, 3000 &DIXMAANI &300, 1500, 3000\\
DIXMAANJ &300, 1500, 3000 &DIXMAANK &300, 1500, 3000\\
DIXMAANL &300, 1500, 3000 &DIXMAANM &300, 1500, 3000\\
DIXMAANN &300, 1500, 3000 &DIXMAANO &300, 1500, 3000\\
DIXMAANP &300, 1500, 3000 &DQRTIC   &100, 500, 1000, 5000\\
EDENSCH &2000& ENGVAL1 &100,1000, 5000\\
FLETCBV2&100&FLETCHCR&100, \textcolor{red}{1000}\\
FMINSRF2 & 121, 961, 1024& FMINSURF& 121, 961, 1024\\
INDEFM  & 100, 1000, 5000& NCB20   & 110, 1010\\
NONCVXU2& 100, 1000, \textcolor{red}{5000} &NONCVXUN& 100, 1000, \textcolor{red}{5000}\\
NONDQUAR &  100, 500, 1000, 5000&PENALTY3&100\\
 POWELLSG & 100, 500, 1000, 5000& POWER   & 100, 500, 1000, 5000\\
 QUARTC   & 100, 500, 1000, 5000&SCHMVETT & 100, 500, 1000, 5000\\
NCB20B & 100,180,500,1000,2000   &SPARSINE & 100, 1000, 5000\\
SPARSQUR  & 100, 1000, 5000&TOINTGSS  & 100, 500, \textcolor{red}{1000, 5000}\\
VARDIM  & 100, 200 & DIXON3DQ &100, 1000\\
DQDRTIC & 100, 500, 1000, 5000&TRIDIA   & 100, 500, 1000, 5000\\
BROYDN7D & 100,500,1000,5000& SINQUAD & \textcolor{red}{100, 500, 1000, 5000}  \\
\end{longtable}

}

 \captionsetup{justification=centering,margin=2cm}

\begin{figure}[H]
\centerline{\includegraphics[width=9cm]{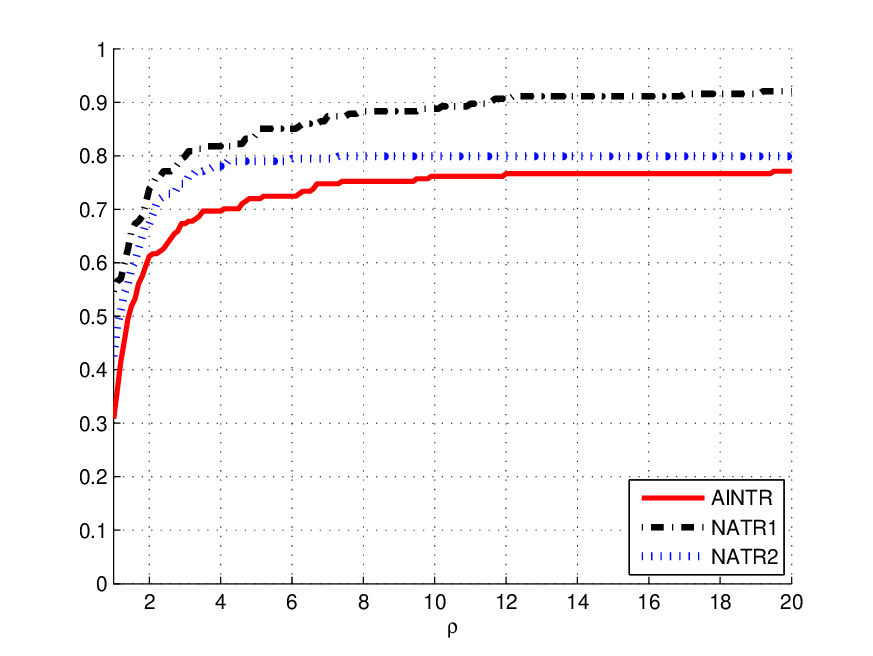}}
\caption{Performance profiles for the number of function evaluations}
\end{figure}

  \begin{figure} [H]
 \centering
   \includegraphics[width=9cm]{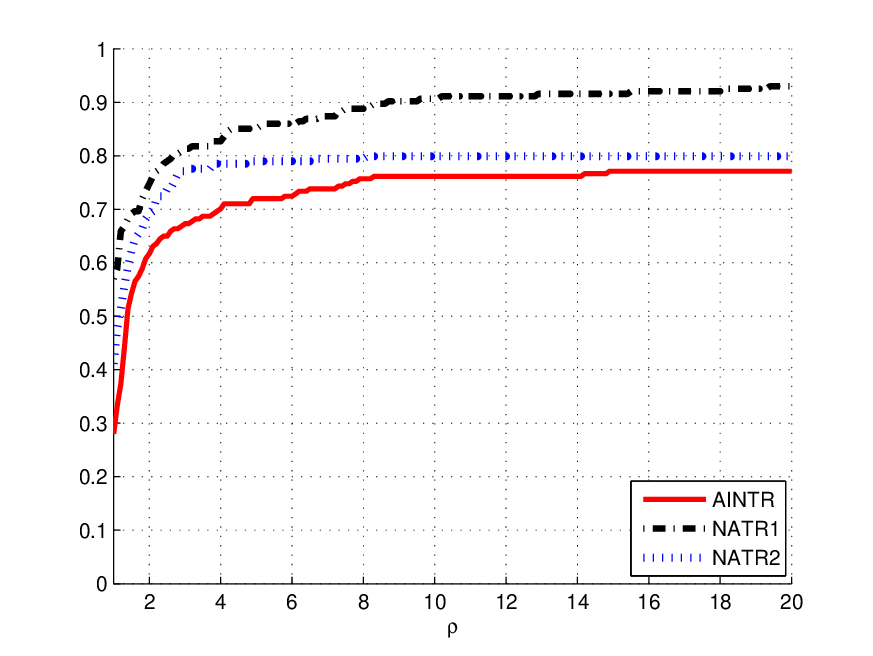}\\
 \caption{Performance profiles for the number of iterations}
\end{figure}

\begin{figure} [H]
\centering
   \includegraphics[width=9cm]{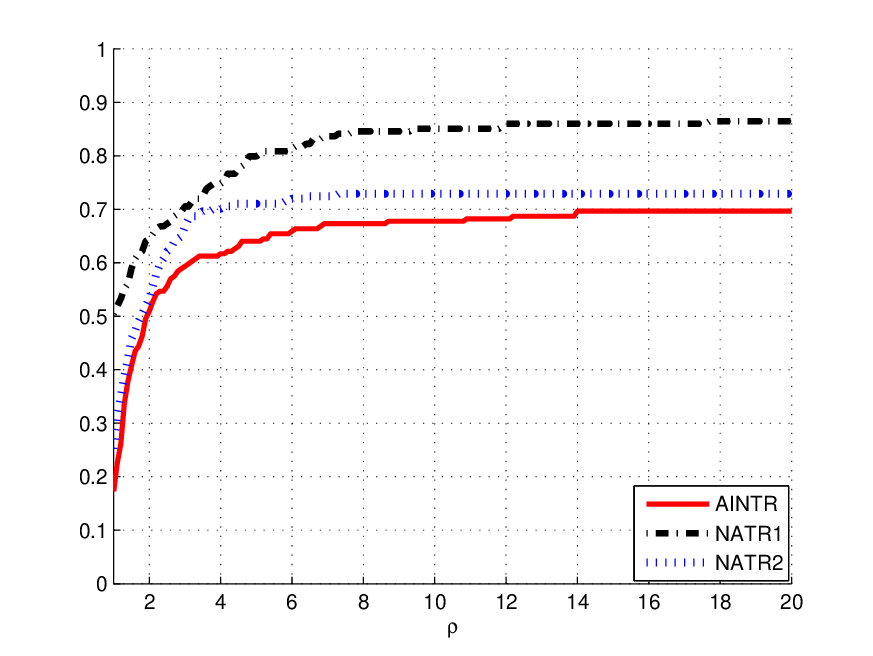}\\
\caption{Performance profiles for the running times}
\end{figure} 

\section{Conclusion}\label{sec5}
In this paper, we propose a new nonmonotone adaptive trust region algorithm to solve unconstrained optimization problems. The new algorithm incorporates a recently proposed adaptive trust region algorithm with nonmonotone techniques. We show that setting a constant shrinkage parameter for the adaptive trust region may impose unnecessary additional computational costs to the algorithm that affects its efficiency. Therefore,  we consider a radius dependent shrinkage parameter in the new algorithm. Further, we propose a new nonmonotone parameter that prevents sudden increments in the objective function values.

 The global convergence of the new algorithm is investigated under some mild conditions. Numerical experiments show the efficiency and robustness of the new algorithm in solving a collection of unconstrained optimization problems from the CUTEst package. It is concluded that exploiting the new ideas is effective to increase the efficiency of the nonmonotone adaptive trust region algorithms and these ideas also can be used  in other nonmonotone and adaptive trust region algorithms which suffer from similar drawbacks mentioned in this paper.

\end{document}